\documentclass{amsart} 

\usepackage{xcolor}

\usepackage[T1]{fontenc} 
\usepackage[USenglish]{babel} 
\usepackage{amsmath,amsfonts,amsthm} 

\usepackage{lipsum} 
\usepackage{graphicx}

\newcommand{\RR}{\mathbb{R}}
\newcommand{\QQ}{\mathbb{Q}}
\newcommand{\ZZ}{\mathbb{Z}}

\newtheorem{thm}{Theorem}[section]
\newtheorem{lemma}[thm]{Lemma}
\newtheorem{prop}[thm]{Proposition}

\newtheorem{rmk}{Remark}
\newtheorem*{mainthm}{Main Theorem}

\numberwithin{equation}{section} 
\numberwithin{figure}{section} 
\numberwithin{table}{section} 

\title{The Distribution of Integers in a Totally Real Cubic Field}

\author{Tianyi Mao} 
\address{Department of Mathematics, CUNY Graduate Center}
\email{tmao@gradcenter.cuny.edu}

\date{\normalsize\today} 

\begin{document}

\maketitle 


\begin{abstract}
Hecke studies the distribution of fractional parts of quadratic irrationals with Fourier expansion of Dirichlet series. This method is generalized by Behnke and Ash-Friedberg, to study the distribution of the number of totally positive integers of given trace in a general totally real number field of any degree. When the field is cubic, we show that the asymptotic behavior of a weighted Diophantine sum is related to the structure of the unit group. The main term can be expressed in terms of Gr{\"o}ssencharacter $L$-functions.
\end{abstract}

\section{Introduction}
The study of the equidistribution of the fractional part of $m\alpha$ for $\alpha$ irrational and $m=1,2,\ldots$ running over the rational integers, dates back to Weyl's work \cite{weyl1910} in 1910. Hecke \cite{hecke1922} studied the case when $\alpha$ is a fixed real quadratic irrational. His key idea is using the Fourier expansion of the Dirichlet series 
\begin{align*}
\sum_{m\ge 1} \left(\{m\alpha\}-\frac{1}{2}\right)m^{-s}
\end{align*}
to estimate the Diophantine sum
\begin{align}\label{s1n}
S_1(n)=\sum_{m=1}^n \left(\{m\alpha\}-\frac{1}{2}\right).
\end{align}

Both Behnke \cite{behnke1923} and Ash-Friedberg \cite{ash2007} aim at generalizing Hecke's result to an arbitrary totally real field $K$ of degree $n$. In such cases, the generalization of the fractional part of $m\alpha$ is the error term in the natural geometric estimate for the number of totally positive integers of $K$ of a given trace. They form the Dirichlet series $\varphi(s)$ whose coefficients are these errors. More specifically, let $O_K$ be the ring of integers of $K$, and let $\text{Tr}(O_K)$ be generated by $\kappa>0$. For positive multiples $a$ of $\kappa$, let $N_a$ denote the number of totally positive integers with trace $a$. There is the natural geometric estimate $r_a$ of $N_a$ derived from the volume of the intersection in $O_K\otimes \RR$ of the cone of totally positive elements with the hyperplane defined by $\text{Tr}\alpha=a$. Denote the difference between the true value and the estimate by $E_a=N_a-r_a$. If $a$ is not a multiple of $\kappa$, we set $E_a=0$. Then we define the Dirichlet series
\[
\varphi(s)=\sum_{a>0} \frac{E_a}{a^s}.
\]

In \cite{behnke1923} and \cite{ash2007}, using the hyperbolic Fourier expansion, $\varphi(s)$ is expressed in terms of an infinite sum involving the Riemann zeta function, Hecke $L$-functions and some Gamma factors. They deduce from this a meromorphic continuation of $\varphi(s)$ in the right half plane $\Re(s)>0$. Each summand is meromorphic on the whole complex plane. However, if $n\ge 3$, the sum will have a dense set of poles on the line $\Re(s)=0$ coming from the Gamma factors, which prevents further analytic continuation. This is an essential difference with Hecke's case $n=2$.

We wish to extract some information on the distribution of the errors $E_a$ from $\varphi(s)$. In \cite{behnke1923} and \cite{ash2007}, the main theorems make conclusions on the asymptotic behaviour of the sum $\sum_{a\le X}E_a$. One way to study the sum is using the Mellin transform and Perron's formula. In other words, we work with the integral
\[
\frac{1}{2\pi i}\int_{n+2- i\infty}^{n+2+ i\infty} \varphi(s) X^s F(s)ds,
\]
for a suitable test function $F(s)$. In the case of Perron's formula, $F(s)=1/s$ is the Mellin transform of the characteristic function of $[0,1]$. We note that if we make $F(s)$ decay fast enough as $\Im(s)\rightarrow \infty$, the integral above will converge absolutely. Then we can move the line of integration and make claims about the corresponding weighted sum $\sum E_a f(a/X)$ where $f(t)$ is the Mellin inversion of $F(s)$.

In particular, for $k\ge 3$, let $f(t)=(-\log(t))^k$ on $[0,1]$, and $f(t)=0$ elsewhere. Then $F(t)=(k!)s^{-k-1}$. For the integral above, we try to move the line of integration through the "dense pole line". We start by writing the integrand as a sum of complex functions that are meromorphic on the whole complex plane. Then we interchange the order of summation and integration, and move the line of integration in each summand. Each summand will contribute a residue term. We use the Stirling's formula and convexity bound to estimate each residue. Their sum is then bounded by a summation over the lattice generated by the inverse of the regulator matrix of $K$, up to some constant. Thus the coordinates of the lattice is some linear combination of logarithms of algebraic numbers. We use Baker's theorem on the poor approximability of such numbers to argue that the sum of the residues is convergent. This require us to limit ourselves to the case when $K$ is a cubic field. The main term comes from the multiple pole at $s=0$. The exact order of that pole depends on the structure of the unit group of the field $K$.

Our main result is:
\begin{mainthm}
Let $K$ be a totally real cubic field. Let $E_a$ be defined as above. For $k\ge 3$,
\begin{align*}
\sum_{n\le X} E_n \log^k\left(\frac{X}{n}\right)&=\frac{3\sqrt{D}}{8\pi^2 (k+1)R} \left(\sum_{(0,v\ne v_0)\text{ good}} L(1,v)\right) (\log X)^{k+1}\\
&+O\left((\log X)^{k}\right).
\end{align*}
(See the following section for the definition of the Hecke $L$-functions and what it means for the pair $(0,v)$ to be good.)
\end{mainthm}

Another direction of generalization of the idea of Hecke is to estimate the average of the partial sum $S_1(n)$ in \eqref{s1n}. For a quadratic irrational $\alpha$, define
\[
M_1(X)=\frac{1}{X}\sum_{n=1}^X S_1(n).
\]
Beck \cite[Prop 6.1]{beck2010} proves that $M_1(X)=c\log X+o(\log X)$, where the constant $c$ depends on the special value of a quadratic Dirichlet $L$-function. He uses the theory of continued fractions. Note that Hardy-Littlewood \cite{hardy1924} also uses continued fractions to give another proof of Hecke's result. It is hard to apply these techniques in our scenario since much less is known about the continued fraction expansions of irrationalities of degree higher than $2$. Other generalizations of Hecke's idea can be found in Duke-Imamoglu \cite{duke2004}, in the case of certain cones; and Zhuravlev \cite{zhuravlev2013}, in case of higher dimensional irrational lattices. We also mention the recent work of B.Borda \cite{Borda2016}. This provides another approach to counting lattice points in more general irrational polytopes.  If the ideas of the present paper can be extended to higher degree number fields, it will be of interest to compare the results to those of \cite{Borda2016}.

The remainder of this paper is organized as follows. In Section 2 we quote the construction of $\varphi(s)$ in \cite{ash2007}. We basically follow their notations and techniques, with some variations to claim that the expression of $\varphi(s)$ is valid in the vertical strip $-n<\Re(s)<0$, and to estimate the growth rate in that region. In Section 3, we express the weighted sum of $E_a$ using the Mellin transform,  and then move the line of integration to obtain the residues plus an error term. Using Baker's theorem, the sum of the residues other than $s=0$ are proved to be of lower order. The residue at $s=0$ contributes a main term, whose coefficient is expressed in terms of Hecke Gr{\"o}ssencharkter $L$-function. The main theorem follows. In Section 4 we give a numerical example for the field of discriminant $257$.

\subsection*{Acknowledgements}
The author acknowledges support from NSF DMS grant 0847586. The computations described in Section 4 were done with the help of the City University of New York High Performance Computing Center at the College of Staten Island, which is supported by National Science Foundation Grants CNS-0958379, CNS-0855217, ACI-1126113.

\section{Dirichlet Series Constructed from Totally Real Number Fields}

Let $K$ be a totally real number field of degree $n$. Ash and Friedberg \cite{ash2007} consider the problem of counting the number of totally positive integers with given trace in $K$. Their method is to consider a family of Dirichlet series constructed from $K$, whose hyperbolic Fourier coefficients turn out to be related to Hecke $L$-functions. In this section, we will follow their notations and state their results without proofs.

Let $\sigma_k :K\rightarrow \RR$, $1\le k\le n$ be the real embeddings of $K$. For $\alpha\in K$, let $\alpha^{(k)}=\sigma_k(\alpha)$. Let $O_K$ denote the integers of $K$, and let $U$ be a subgroup of the units of $O_K$ generated by all totally positive units and $-1$. Let $u_l$, $1\le l\le n-1$, be units which together with $\{\pm 1\}$ generate $U$. Let $V=\RR^{n-1}$, and let $\Lambda_U$ be the lattice in $V$ spanned by the vectors
\[
\lambda_l=\left(\log|u_l^{(1)}|-\log|u_l^{(n)}|,\log|u_l^{(2)}|-\log|u_l^{(n)}|,\ldots,\log|u_l^{(n-1)}|-\log|u_l^{(n)}| \right).
\]
Let $\Lambda_U^*$ be the dual lattice to $\Lambda_U$ in $V$ with respect to the standard inner product.

For each $i$, $1\le i\le n$, choose $e_i=0$ or $1$, and let $v(\alpha)=\prod_{i=1}^n \text{sgn}(\alpha^{(i)})^{e_i}$. Let $v_0$ correspond to $e_i=0$ for all $i$. We call a pair $(\mu,v)$ good if
\[
\prod_{j=1}^{n-1}\left|\frac{u^{(n)}}{u^{(j)}}\right|^{-2\pi i\mu_j} v(u)=1
\]
for all $u\in O_K^\times$. For such a pair, let
\[
\lambda_{\mu,v}(\alpha)=\prod_{j=1}^{n-1}\left|\frac{\alpha^{(n)}}{\alpha^{(j)}}\right|^{-2\pi i\mu_j} v(\alpha).
\]
Then $\lambda_{\mu,v}$ can be extended to a Hecke character. Define the partial Hecke $L$-function
\[
L(s,\lambda_{\mu,v})=\sum_{(\beta)\ne (0)}\lambda_{\mu,v}((\beta))N(\beta)^{-s}
\]
where the sum is over all nonzero principal ideal of $O_K$.

Define 
\begin{align*}
\Psi(s,v)=\sum_{0\ne \alpha\in O_K} \frac{v(\alpha)}{(|\alpha^{(1)}|+\ldots+|\alpha^{(n)}|)^s}.
\end{align*}
This sum converges for $\Re(s)>n$. We have the following theorem:
\begin{prop}\label{prop21}
Each function $\Psi(s,v)$ has meromorphic continuation to the right half plane $\Re(s)>0$. Explicitly, we have
\begin{align}\label{defpsi}
\Psi(s,v)=\sum_{\mu\in \Lambda_U^*} a_\mu(s,v)
\end{align}
where $a_\mu(s,v)=0$ unless $(\mu,v)$ is good. In that case we have
\[
a_\mu(s,v)=\frac{2}{nR}\frac{1}{\Gamma(s)} L\left(\frac{s}{n},\lambda_{\mu,v}\right)\prod_{j=1}^n \Gamma\left(\frac{s}{n}-2\pi i\mu_j\right).
\]
Furthermore, the functions $\Psi(s,v)$ are holomorphic in this right half plane for $v\ne v_0$, while $\Psi(s,v_0)$ has a simple pole at $s=n$ of residue
\[
\frac{2^n}{(n-1)!\sqrt{D}}
\]
where $D$ is the discriminant of $K$.
\end{prop}

\begin{proof}
See \cite{ash2007}, Prop 4.1.
\end{proof}

Remark: The proof in \cite{ash2007} actually shows that the sum in \eqref{defpsi} is well-defined in any bounded vertical strip that contains no poles from the gamma factors in $a_\mu(s,v)$. In Section $8$ of \cite{ash2007} , this is proved quantitatively by studying the growth rate of $\Psi(s,v)$ in the strip $0<\Re(s)<n$, as $\Im(s)\rightarrow\infty$. We will imitate this process later.

Let $N_a$ denote the number of totally positive integers with trace $a$. One can obtain an approximation of $N_a$ by a geometric estimate, called $r_a$, given by
\[
r_a=\frac{ka^{n-1}}{(n-1)!\sqrt{D}},
\]
where $\kappa>0$ is the generator of $\text{Tr}(O_K)$. (See \cite{ash2007}, Prop 5.1.)

Ash and Friedberg study the Dirichlet series
\begin{align*}
\varphi(s)=\sum_{a>0}\frac{E_a}{a^s}=\sum_{a>0}\frac{N_a-r_a}{a^s}.
\end{align*}
The second part
\[
Z(s)=\sum r_a a^{-s}=\frac{\kappa^{n-s}}{(n-1)!\sqrt{D}}\zeta(s-n+1)
\]
is just a translation of Riemann zeta function. The first part can be written as
\begin{align*}
\sum_{a>0}\frac{N_a}{a^s}=2^{-n}\sum_v \Psi(s,v),
\end{align*}
where $v$ is over all possible combinations $v(\alpha)=\prod_{i=1}^n \text{sgn}(\alpha^{(i)})^{e_i}$, each $e_i=0$ or $1$. We rewrite
\begin{align}
\varphi(s)=2^{-n}\sum_v \Psi(s,v)-Z(s)\label{phidef}
\end{align}
for further convenience.

When $n=2$, $\varphi(s)$ can be meromorphically continued to the whole complex plane. However, when $n\ge 3$, this cannot be done because we have a dense set of poles on the line $\Re(s)=0$, coming from the Gamma factors in $\Psi(s,v)$. However, in the strip $-n<\Re(s)<0$, we can define
\begin{align*}
\Psi_0 (s,v)&=\sum_{\mu\in \Lambda_U^*} a_\mu(s,v)\\
\varphi_0(s)&=2^{-n}\sum_v \Psi_0 (s,v)-Z(s).
\end{align*}
We have
\begin{prop}\label{prop22}
The functions $\Psi_0 (s,v)$ and $\varphi_0(s)$ are holomorphic functions in the strip $-n<\Re(s)<0$. Given $\epsilon,\delta>0$, we have
\[
\varphi_0 (s)=O(|t|^{n-\frac{1}{2}-\sigma+\epsilon})
\]
as $t\rightarrow \infty$ uniformly for $s=\sigma+it$ with $-n+\delta\le \sigma\le -\delta$, with the implied constant depending on $\epsilon$ and $K$.
\end{prop}

\begin{proof}
We use the same argument as in \cite{ash2007}, Prop 8.1. 

An estimate of this form for $\zeta(s-n+1)$ follows immediately from the functional equation.

For $s=\sigma+it$, we have
\[
L\left(\frac{s}{n},\lambda_{\mu,v}\right)\ll \prod_{j=1}^n \left|\frac{t}{n}-2\pi \mu_j \right|^{1/2-\sigma/n+\epsilon}
\]
as $|t|\rightarrow \infty$ in the strip $-n<\Re(s)<0$, simply by the functional equation. Combining with Stirling's formula, one sees that
\[
a_\mu(s,v)\ll |t|^{1/2-\sigma} e^{\pi |t|/2} \prod_{j=1}^n \left|\frac{t}{n}-2\pi \mu_j \right|^{\epsilon} e^{-\frac{\pi}{2}\left|\frac{t}{n}-2\pi \mu_j \right|}.
\]
Let
\[
F(s,\mu)=\prod_{j=1}^n \left|\frac{t}{n}-2\pi \mu_j \right|^{\epsilon} e^{-\frac{\pi}{2}\left|\frac{t}{n}-2\pi \mu_j \right|}.
\]
Then
\[
\sum_{\mu\in \Lambda_U^*} |a_\mu(s,v)| \ll |t|^{1/2-\sigma} e^{\pi |t|/2}\sum_\mu F(s,\mu).
\]

To complete the proof, we use the following lemma from \cite{ash2007}:
\begin{lemma}[\cite{ash2007}, Prop 8.1]
Let
\[
F_1(T,\eta)=\prod_{j=1}^n |T-\eta_j|^\beta e^{-|T-\eta_j|},
\]
where $\beta>0$, $T>0$, and $\eta=(\eta_1,\ldots,\eta_n)$. Suppose $\Lambda$ is a lattice in the hyperplane $H=\{x|x_1+\ldots+x_n=0\}\subset \RR^n$, then the sum
\[
\sum_{\eta\in\Lambda} F_1(T,\eta)\ll e^{-nT} T^{n-1+n\beta}.
\]
The implicit constant depending on $\beta$ and $K$.
\end{lemma}

Now we let $T=\frac{\pi t}{2n}$, $\beta=\epsilon$, $\eta_j=\pi^2 \mu_j$, and we will end up with
\[
\Psi_0 (s,v)\ll \sum_\mu |a_\mu(s,v)|\ll |t|^{n-\frac{1}{2}-\sigma+n\epsilon}.
\]
\end{proof}

\section{Working with the Cubic Field}

From now on we assume $K$ is a totally real cubic field. We want to further study the behaviour of $E_a$. Recall that the corresponding Dirichlet series for $E_a$ is
\begin{align*}
\varphi(s)=\sum_{a>0}\frac{E_a}{a^s}=2^{-n}\sum_v \Psi(s,v)-Z(s)
\end{align*}
as in \eqref{phidef}.

A usual way to study an arithmetic quantity from its Dirichlet series is using the Mellin transform. For an integer $k\ge 3$, let $f(t)=(-\log(t))^k$ on $[0,1]$, and $f(t)=0$ elsewhere. Let $F(t)=(k!)s^{-k-1}$ be its Mellin transform. To the right of the abscissa of absolute convergence of $\varphi(s)$, we have
\begin{align*}
S(X)=\sum_{n\ge 1} E_n f\left(\frac{n}{X}\right)&=\frac{1}{2\pi i}\int_{5- i\infty}^{5+ i\infty} \varphi(s) X^s F(s)ds.
\end{align*}

The usual technique is to move the line of integration and use residue theorem. We wish to move beyond the "wall of dense poles" on $\Re(s)=0$. To do this, we change the order of summation and integration. Using \eqref{phidef}, for small $\delta>0$,
\begin{align*}
&\frac{1}{2\pi i}\int_{5- i\infty}^{5+ i\infty} \varphi(s) X^s F(s)ds
=\frac{1}{2\pi i}\int_{\delta- i\infty}^{\delta+ i\infty} \varphi(s) X^s F(s)ds\\
=&\frac{1}{2\pi i}\int_{\delta- i\infty}^{\delta+ i\infty} \left(\frac{1}{8}\sum_v \Psi(s,v)-Z(s)\right) X^s F(s)ds\\
=&\frac{1}{8}\sum_v \frac{1}{2\pi i}\int_{\delta- i\infty}^{\delta+ i\infty} \Psi(s,v) X^s F(s)ds - \frac{1}{2\pi i}\int_{\delta- i\infty}^{\delta+ i\infty} Z(s) X^s F(s)ds\\
=&I_1-I_2
\end{align*}
We continue by using Proposition \ref{prop21} to expand $\Psi(s,v)$ in $I_1$:
\begin{align}
I_1&=\frac{1}{12R}\sum_v\sum_\mu \frac{1}{2\pi i}\int_{\delta- i\infty}^{\delta+ i\infty} \Gamma(s)^{-1} L\left(\frac{s}{3},\lambda_{\mu,v}\right)\prod_{j=1}^3 \Gamma\left(\frac{s}{3}-2\pi i\mu_j\right) X^s F(s)ds \label{integrand}
\end{align}
Now for each integral, the number of poles on the imaginary axis is at most $3$. Thus we can move the line of integration from $\Re(s)=\delta$ to $\Re(s)=-\delta$. Thus
\begin{align}
I_1&=\frac{1}{12R}\sum_{\mu,v}\left( \sum_{j=1}^3\text{Res}_{\mu,v,j}\right.\\
&+\left.\frac{1}{2\pi i}\int_{-\delta- i\infty}^{-\delta+ i\infty} \Gamma(s)^{-1} L\left(\frac{s}{3},\lambda_{\mu,v}\right)\prod_{j=1}^3 \Gamma\left(\frac{s}{3}-2\pi i\mu_j\right) X^s F(s)ds\right)\notag\\
&=\frac{1}{8}\sum_v \frac{1}{2\pi i}\int_{-\delta- i\infty}^{-\delta+ i\infty} \Psi_0(s,v) X^s F(s)ds+\frac{1}{12R}\sum_{\mu,v,j}\text{Res}_{\mu,v,j}\notag\\\notag
&=\frac{1}{2\pi i}\int_{-\delta- i\infty}^{-\delta+ i\infty} \varphi_0(s) X^s F(s)ds+\frac{1}{2\pi i}\int_{-\delta- i\infty}^{-\delta+ i\infty} Z(s) X^s F(s)ds\\
&+\frac{1}{12R}\sum_{\mu,v,j}\text{Res}_{\mu,v,j}\notag\\
&=I_3+I_4+\frac{1}{12R}\sum_{\mu,v,j}\text{Res}_{\mu,v,j}\label{ressum}
\end{align}
where $\text{Res}_{\mu,v,j}$ is the residue of $\Gamma(s)^{-1} L\left(\frac{s}{3},\lambda_{\mu,v}\right)\prod_{\ell=1}^3 \Gamma\left(\frac{s}{3}-2\pi i\mu_\ell\right) X^s F(s)$ at $s=6\pi i\mu_j$, if such a pole exist.

The integrals $I_2, I_4$ are dealt with more easily. By using the fact that $Z(0)=0$, $Z(s) X^s s^{-k-1}$ has a pole of order at most $k$ at $s=0$, we have $I_4-I_2=O\left((\log X)^{k-1}\right)$. Also note that, by Proposition \ref{prop22}, for all $k\ge 3$, $F(s)=(k!)s^{-k-1}$ is able to cancel the growth of $\varphi_0(s)$ in a vertical line so that $I_3$ converges absolutely. This means $I_3=O(X^{-\delta})$.

Before we analyze the residues, we first need to prove that
\begin{lemma}
Let $K$ be a totally real cubic field, and $\mu=(\mu_1,\mu_2)\in \Lambda_U^*$. Let $\mu_3=-\mu_1-\mu_2$. Suppose $\mu_i=\mu_j$ for some $i\ne j$, then $\mu=0$.
\end{lemma}
\begin{proof}
If $\mu_1=\mu_2$, then since $\langle \mu,\lambda_1\rangle,\langle \mu,\lambda_2\rangle \in \ZZ$, we have
\begin{align*}
\mu_1 \left(\log \left| \frac{u_1^{(1)}u_1^{(2)}}{(u_1^{(3)})^2} \right|\right)=a\\
\mu_1 \left(\log \left| \frac{u_2^{(1)}u_2^{(2)}}{(u_2^{(3)})^2} \right|\right)=b
\end{align*}
for some integers $a,b$. If $\mu_1\ne 0$, then
\begin{equation*}
\log \left| \frac{u_1^{(1)}u_1^{(2)}}{(u_1^{(3)})^2} \right|^b \left| \frac{u_2^{(1)}u_2^{(2)}}{(u_2^{(3)})^2} \right|^{-a} =0
\end{equation*}
Let $u=u_1^b u_2^{-a}$, then $u\in O_K^\times$, and
\begin{equation*}
\left| \frac{u^{(1)}u^{(2)}}{(u^{(3)})^2}\right|=1
\end{equation*}
Using the fact that $\left|N(u)\right|=1$, we have $\left|u^{(3)}\right|=1$, so $a=b=0$. But that implies
\begin{equation*}
\left| \frac{u_j^{(1)}u_j^{(2)}}{(u_j^{(3)})^2}\right|=1
\end{equation*}
for $j=1,2$, so $\left|u_j^{(3)}\right|=1$, which contradicts with the fact that $u_1,u_2$ generate $U$ with $-1$.

The same argument applies if $\mu_1=\mu_3$ or $\mu_2=\mu_3$.
\end{proof}

We now return to the integrand of \eqref{integrand}. Except the pole at $s=0$, we only have at most simple poles at
\[
s=6\pi i\mu_j,j=1,2\text{ or }3
\]
and the residues are (if there is a pole)
\begin{align*}
&\text{Res}_{\mu,v,j}=\Gamma(6\pi i\mu_j)^{-1} L(2\pi i\mu_j,\lambda_{p,\mu})\prod_{i\ne j}\Gamma(2\pi i(\mu_j-\mu_i))X^{6\pi i\mu_j} (1+6\pi i\mu_j)^{-k-1} k!\\
&\ll \exp(3\pi^2|\mu_j|)|\mu_j|^{1/2}\left(\prod_{i\ne j}|\mu_j-\mu_i|^{1/2+\epsilon}\right)\left(\prod_{i\ne j}\exp(-\pi^2 |\mu_j-\mu_i|)|\mu_j-\mu_i|^{-1/2}\right) \\
& \cdot |\mu_j|^{-k-1}\\
&=\exp(\pi^2(3|\mu_j|-|\mu_j-\mu_{i_1}|-|\mu_j-\mu_{i_2}|)\left(\prod_{i\ne j}|\mu_j-\mu_i|^{\epsilon}\right)|\mu_j|^{-k-1/2}
\end{align*}
by Stirling's formula and convexity bound of $L$-function. Note that $\mu_1+\mu_2+\mu_3=0$. Let $\nu_1=\mu_j-\mu_{i_1}$, $\nu_2=\mu_j-\mu_{i_2}$, then $3\mu_j=\nu_1+\nu_2$, and the bound above can be rewritten as
\[
\exp\left(-\pi^2\left(|\nu_1|+|\nu_2|-|\nu_1+\nu_2|\right)\right)|\nu_1|^\epsilon |\nu_2|^\epsilon |\nu_1+\nu_2|^{-k-1/2}.
\]
We wish to study the convergence of this bound summed over an appropriate lattice of $\nu$ in $\RR^2$. 

As $\mu$ runs over the lattice $\Lambda_U^*$, $\nu=(\nu_1,\nu_2)$ also runs over a lattice $\Lambda '$ in $\RR^2$. We will prove 
\begin{prop}\label{prop32}
The sum
\[
\sum_{0\ne\nu\in \Lambda '}\exp\left(-\pi^2\left(|\nu_1|+|\nu_2|-|\nu_1+\nu_2|\right)\right)|\nu_1|^\epsilon |\nu_2|^\epsilon|\nu_1+\nu_2|^{-k-1/2}
\]
is convergent. 
\end{prop}
\begin{proof}
When $\nu_1 \nu_2>0$, the exponential term is $1$, and we may assume without loss of generality that $\nu_1,\nu_2>0$. We group the points $\nu$ by the regions
\[
P_N=\{\nu_1,\nu_2>0 : N\le \nu_1+\nu_2< N+1\}.
\]
Note that in $P_0$, $|\nu_1+\nu_2|$ is bounded below by a constant $C_0=C_0(K)$ which only depends on the field $K$. In $P_N$ with $N\ge 1$, $|\nu_1+\nu_2|$ is bounded below by $N$. Recall that the number of lattice points in a convex compact set with diameter $d$ is bounded roughly by $d^2$. Since we can divide each strip $P_N$ into $2N+1$ unit triangles with diameter $\sqrt{2}$, we have $\text{card}(\Lambda '\cap P_N)\ll N$ (By \cite{ash2007}, Lemma 7.1). Thus
\[
\sum_{0\ne\nu\in \Lambda ',\nu_1\nu_2>0}|\nu_1+\nu_2|^{-k-1/2+2\epsilon}\ll \sum_{N\ge 1}N^{-k+1/2+2\epsilon}<\zeta\left(k-\frac{1}{2}-2\epsilon\right)<\infty
\]
since $k\ge 3$.

Now we look at the part of the sum where $\nu_1\nu_2<0$. Here the exponential term plays an important part. Again without loss of generality, assume $\nu_1>0$, $\nu_2<0$ and $\nu_1+\nu_2>0$. In this case,
\[
|\nu_1|+|\nu_2|-|\nu_1+\nu_2|=\nu_1-\nu_2-(\nu_1+\nu_2)=-2\nu_2=2|\nu_2|.
\]
Consider the points in the regions
\[
Q_{M,N}=\{\nu_1>0,\nu_2<0,\nu_1+\nu_2>0 : N\le |\nu_1+\nu_2|< N+1, M\le |\nu_2|<M+1\}.
\]
Each region has bounded diameter, so the number of lattice points inside it is bounded by a constant. For $N\ge 1$, $|\nu_1+\nu_2|$ is bounded from below by $N$. Thus the sum in $Q_{M,N}$ is bounded from above by a constant times
\[
\exp(-2M)(2+M+N)^\epsilon (M+1)^\epsilon N^{-k-1/2}.
\]
Summing this over the regions $Q_{M,N}$ with $M\ge 0$ and $N\ge 1$, the sum is bounded by
\[
\sum_{M\ge 0,N\ge 1}\exp(-2M) (2+M+N)^\epsilon (M+1)^\epsilon N^{-k-1/2}<\infty.
\]
Finally, consider the case $N=0$. We wish to prove that in $Q_{M,0}$, the quantity $|\nu_1+\nu_2|$ is bounded from below by some function of $M$. Fix a basis $b_1,b_2$ of $\Lambda'$. Let $Q_M$ be the parallelogram defined by
\[
Q_M=\{a_1 b_1+a_2 b_2 : |a_i|\le M\}.
\]
Note that there exists a constant $C_1=C_1(K,b)$ depending only on the field $K$ and the basis $b_1,b_2$, such that $Q_{M,0}$ is inside the dilation $C_1\cdot Q_M$. Thus $|\nu_1+\nu_2|$ is a linear combination of the entries of $b_1$ and $b_2$ with coefficients bounded by $C_1 M$. Note that the entries come from the inverse of the regulator matrix of $K$ (possibly after some linear combinations), which consists of logarithm of the units, up to a constant only dependent on $K$. Now we use the following theorem by Baker \cite{baker1967}:
\begin{thm}[\cite{baker1967}, Theorem 3]
Let $n\ge 2$ be an integer, and let $\alpha_1,\ldots, \alpha_n$ denote non-zero algebraic numbers such that $\log \alpha_1, ..., \log \alpha_n$ are linearly independent over the rationals. Further suppose $\kappa>2n+1$, and let d be any positive integer. There is an effectively computable number
\[
C = C(n, \alpha_1, \ldots, \alpha_n,\kappa,d) > 0
\]
such that for all algebraic numbers $\beta_1,\ldots,$ $\beta_n$, not all $0$, with degree at most $d$, we have
\[
|\beta_1\log\alpha_1+\ldots+\beta_n\log\alpha_n|>CH^{-\kappa}
\]
where $H$ is the maximum of the heights of the $\beta_i$'s.
\end{thm}
Baker's theorem shows that $|\nu_1+\nu_2|\gg M^{-\kappa}$. Hence the sum over all $Q_{M,0}$ are bounded from above by
\[
\sum_{M\ge 1} \exp(-2M) (M+3)^{2\epsilon} M^{\kappa(k+1/2)}<\infty.
\]
This concludes the proof of Proposition \ref{prop32}.
\end{proof}

Putting the above together, we have proved that the sum $\sum_{\mu\ne 0}\text{Res}_{\mu,v,j}$ over residues in \eqref{ressum} is convergent for any $v$ and $j$.

Now we turn to $\mu=0$, which possibly contributes a higher order pole at $s=0$ because of the gamma factors. Isolating that part and replacing $\frac{s}{3}$ with $s$, we look at
\begin{align*}
J_{\mu,v}(s)=L(s,\lambda_{\mu,v})\prod_{j=1}^3 \Gamma(s-2\pi i \mu_j)
\end{align*}
at $s=0$. Note that the gamma factor for the functional equation of $L(s,\lambda_{\mu,v})$ is
\begin{align*}
\Gamma(s,\lambda_{\mu,v})=(\pi^{-3} D)^{s/2}\prod_{j=1}^3 \Gamma\left(\frac{s}{2}+\frac{e_j}{2}-\pi i \mu_j\right).
\end{align*}
Let
\begin{align*}
L^*(s,\lambda_{\mu,v})=L(s,\lambda_{\mu,v})\Gamma(s,\lambda_{\mu,v})
\end{align*}
be the completed partial Hecke $L$-function. Using the duplication formula we find that
\begin{align*}
J_{\mu,v}(s)=L^*(s,\lambda_{\mu,v})(\pi^{-3} D)^{s/2} \left(\frac{2^{s-1}}{\sqrt{\pi}}\right)^3 \prod_{j=1}^3 \Gamma\left(\frac{s}{2}+\frac{1-e_j}{2}-\pi i \mu_j\right)
\end{align*}
When $(\mu,v)=(0,v_0)$, the function $L^*(s,\lambda_{0,v_0})$ has a simple pole at $s=0$. In that case $J_{0,v_0}(s)$ has only a simple pole at $s=0$. The higher order terms come from the gamma factors when all $\mu_i$'s are zero, and some of the corresponding $e_i$'s are $1$.

For further convenience, let
\[
L(s,v)=L(s,\lambda_{0,v}), L^*(s,v)=L^*(s,\lambda_{0,v})
\]
We shall show
\begin{prop}
Let $K$ be a cubic, totally real field with discriminant $D$. When $v\ne \text{id}$ and $(0,v)$ is good, $J_{0,v}(s)$ has a double pole at $s=0$ with leading Laurent coefficient
\begin{equation*}
\frac{\sqrt{D}}{2\pi^2} L(1,v) \frac{1}{s^2}.
\end{equation*}
All other poles of $J_{0,v}(s)$ on the line $\Re (s)=0$ are simple.
\end{prop}

\begin{proof}
In the cubic case, for any good pair $(0,v)$ with $v\ne v_0$, $L(1,v)$ is a nonzero finite number, so is $L^*(1,v)=L^*(0,v)$, by the functional equation. Also note that two of the $e_i$'s are $1$, and the other one is $0$. Thus
\begin{align*}
J_{\mu,v}(s)=L^*(s,\lambda_{\mu,v})(\pi^{-3} D)^{s/2} \left(\frac{2^{s-1}}{\sqrt{\pi}}\right)^3 \prod_{j=1}^3 \Gamma\left(\frac{s}{2}+\frac{1-e_j}{2}-\pi i \mu_j\right)
\end{align*}
has a double pole at $s=0$, and the Laurent expansion can be explicitly calculated by the following expression:
\begin{align*}
J_{0,v}(s)&=\frac{1}{8 \pi}L^*(s,v)\Gamma\left(\frac{s}{2}\right)^{2}\\
&=\frac{1}{8 \pi}L^*(s,v)\left( \frac{4}{s^2}-\frac{4\gamma}{s}+\ldots \right).
\end{align*}
Use the functional equation of $L^*$ to rewrite
\begin{align*}
L^*(0,v)&=L^*(1,v)=\frac{\sqrt{D}}{\pi}L(1,v).
\end{align*}
We conclude that the Laurent series of $J_{0,v}(s)$ has the form
\begin{equation*}
\frac{\sqrt{D}}{2\pi^2} L(1,v) \frac{1}{s^2}+O\left(\frac{1}{s}\right)\text{ as }s\rightarrow 0.
\end{equation*}

Also note that for $\mu\ne 0$, $J_{p,\mu}(s)$ have poles at $s={e_j-1}+2\pi i \mu_j$. When $e_j=1$, they will lie on the line $\Re (s)=0$. To show that such poles are simple, we need to prove that $\mu_i\ne \mu_j$ for $i\ne j$. This is exactly the statement of Lemma 3.1.
\end{proof}


Now we know exactly how the residues on $s=0$ will contribute. If there exists $v$ so that $(0,v)$ is good, then the integrand
\begin{align}\label{mainintegrand}
\frac{1}{12R}\frac{1}{\Gamma(s)} L\left(\frac{s}{3},v\right) \Gamma\left(\frac{s}{3}\right)^3 X^s (k!)s^{-k-1}
\end{align}
has a pole of order $k+2$. When we move the line of integration, its residue will contribute
\[
\frac{3\sqrt{D}}{8\pi^2 (k+1)R} L(1,v) (\log X)^{k+1}+O\left((\log X)^{k}\right)
\]

Now we may conclude that
\begin{thm}
For any integer $k\ge 3$,
\begin{align*}
\sum_{n\le X} E_n \log^k\left(\frac{X}{n}\right)&=\frac{3\sqrt{D}}{8\pi^2 (k+1)R} \left(\sum_{(0,v\ne v_0)\text{ good}} L(1,v)\right) (\log X)^{k+1}\\
&+O\left((\log X)^{k}\right)
\end{align*}
\end{thm}

\begin{rmk}
For a nontrivial $(0,v)$ to be good, we need to be able to find a product of two real embeddings $\sigma_i, \sigma_j$ , such that $\sigma_i(u) \sigma_j(u)>0$ for all $u\in O_K^\times$. For totally real cubic fields with discriminant less than $1000$, there are only $4$ fields with exactly one nontrivial good $(0,v)$, whose discriminants are $257,697,788,985$. For other fields, the sum over good nontrivial $(0,v)$ is empty.
\end{rmk}

\section{A Numerical Computation}

In this section, we will provide a numerical example on counting the number of totally positive integers in a specific cubic field, and compare the results with our main theorem. The computations described
here were done in Sage and Julia (to enumerate lattice points) and Magma (to compute the special values of the $L$-functions involved).

Let $K=\QQ(\alpha)=\QQ(x)/(x^3+2x^2-3x-1)$ be the totally real cubic field with discriminant $257$. We fix an integral basis of $O_K$
\[
\beta_1=2+3\alpha, \beta_2=5\alpha+\alpha^2, \beta_3=1+\alpha.
\]
Recall that
\[
\beta_j^{(i)}, i=1,2,3
\]
are the real embeddings of $\beta_j$. Note that $\text{Tr}\beta_1=\text{Tr}\beta_2=0$, $\text{Tr}\beta_3=1$. For any element $z\in O_K$ with $\text{Tr}(z)=a$, we can write
\[
z=c_1\beta_1+c_2\beta_2+a\beta_3,
\]
where $(c_1,c_2)\in \ZZ^2$. Regarding $c_1,c_2$ as the variables, we have that $N_a$, the number of totally positive integers in $O_K$ with trace $a$, is equal to the number of lattice points in the interior of the triangle $T_a$ formed by the lines
\[
c_1\beta_1^{(i)}+c_2\beta_2^{(i)}+a\beta_3^{(i)}=0, i=1,2,3.
\]
One can verify that $T_a=aT_1$, and $r_a=\frac{a^2}{2\sqrt{257}}$ is the area of $T_a$. To get the exact value of $N_a$, one can solve the coordinates of the triangle $T_a$ and do a line sweep with a computer program. 

\begin{figure}[h]
\includegraphics[width=.65\textwidth]{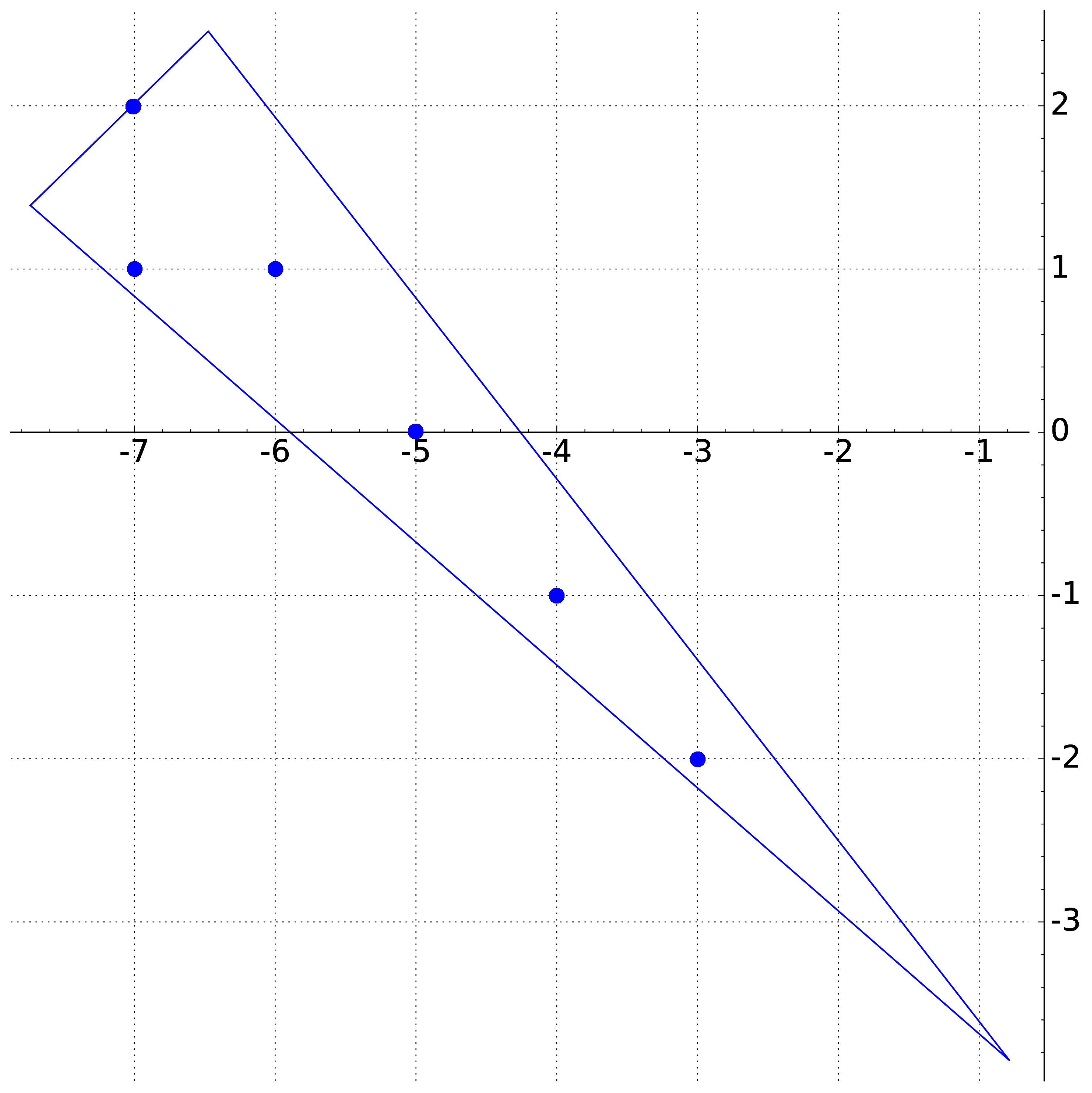}
\caption{Lattice points inside the triangle $T_{15}$}
\end{figure}


Now we turn to the theoretical side. Note that, with more detailed calculations on the Laurent series of the integrand in \eqref{mainintegrand}, one can write out more asymptotic terms explicitly for the sum 
\[
S(X)=\sum_{n\le X} E_n \log^k\left(\frac{X}{n}\right)
\]
\begin{figure}[h]
\includegraphics[width=\textwidth]{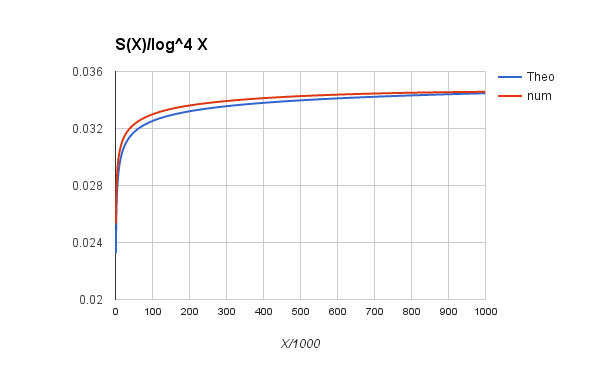}
\caption{Numerical Value of $S(X)/\log^4 X$, up to $X=10^6$}
\end{figure}
in the main theorem, up to $\log X$. For simplicity, We do not include the algebraic expressions of the coefficients here. By numerical computation, the first three terms are:
\[
S(X) \approx 0.041983745 \log^4 X -0.07792862 \log^3 X -0.35634540 \log^2 X+O(\log X).
\]

With the data of $N_a$, we computed $S(X)/\log^4 X$ up to $X=10^6$. It turns out that the quotient is increasing very slowly in this range. Figure 4.2 shows good agreement between numerical data
and the first three terms of the asymptotic above.


\bibliographystyle{plain}
\bibliography{logbib}


\end{document}